\documentclass[a4paper,8pt]{article}
\usepackage[english]{babel}
\usepackage{geometry}
\usepackage[T1]{fontenc}
\usepackage[latin1]{inputenc}
\usepackage{amsmath,amssymb,mathrsfs,amsthm}
\usepackage{soul}
\usepackage{epsfig}
\usepackage{hyperref}
\usepackage{graphicx}
\usepackage{xypic}
\usepackage{datetime}
\usepackage{fancyhdr}
\fancypagestyle{plain}{
\lhead{}
\rhead{}
\lfoot{}

}

\usepackage{tocloft}

\newtheorem{theorem}{Theorem}[section]
\newtheorem{proposition}[theorem]{Proposition}

\newtheorem{lemma}[theorem]{Lemma}

\newtheorem{Corollaire}[theorem]{Corollary}
\newtheorem{definition}[theorem]{Definition}

\newtheorem{remarque}[theorem]{Remark}

\newcommand{\vc}{\|\cdot\|}
\newcommand{\C}{\mathbb{C}}
\newcommand{\mc}{\mathcal{O}(D)}

\newcommand{\Si}{\Sigma}

\newcommand{\al}{\alpha}

\newcommand{\R}{\mathbb{R}}
\newcommand{\cl}{\mathcal{C}^\infty}
\newcommand{\p}{\mathbb{P}}
\newcommand{\eps}{\varepsilon}

\newcommand{\vf}{\varphi}

\newcommand{\N}{\mathbb{N}}
\newcommand{\z}{\overline{z}}
\newcommand{\pt}{\partial}

\newcommand{\m}{\omega}

\newcommand{\e}{\textbf{e}}
\title{{\scshape{Growth of balls of holomorphic sections on projective toric varieties}}}
\date{}
\author{Mounir Hajli \footnote{\small Institute of Mathematics, Academia Sinica,
Taipei 106, Taiwan \quad\emph{E-mail}:\ttfamily{hajli@math.sinica.edu.tw, hajlimounir@gmail.com}}}
\begin{document}

\maketitle

\begin{abstract} 
 Let $\mathcal{O}(D)$ be an equivariant line bundle which is big and nef on a complex projective nonsingular toric variety $X$. 
Given a continuous toric  metric $\|\cdot\|$ on $\mathcal{O}(D)$, we define the energy at equilibrium of 
$(X,\phi_{\overline{D}})$ where $\phi_{\overline{D}}$ is the weight of  the metrized toric divisor $\overline{D}=(D,\|\cdot\|)$. We show that this energy describes the asymptotic behaviour 
as $k\rightarrow \infty$ of the volume of the $L^2$-norm unit ball induced by $(X,k\phi_{\overline{D}})$ on the space of global holomorphic sections $H^0(X,\mathcal{O}(kD))$.

\end{abstract}

\section{{\scshape{Introduction}}}

Let $Q$ be a free $\mathbb{Z}$-module of rank $n$ and $P$ its  dual. We  consider a fan $\Sigma$ on $Q_\mathbb{R}=
 Q\otimes_\mathbb{Z} \R$ and we denote by $X=X_\Si$  the  associated  toric variety over $\mathbb{C}$, see for instance \cite{Oda}. In the sequel, we assume
 that $X$ is    nonsingular and projective (this is equivalent to the fact that  $\Sigma$ is nonsingular and
  the support of $\Sigma$ is
  $Q_\mathbb{R}$, see \cite[theorems 1.10, 1.11]{Oda}). We  set $\mathbb{T}_Q:=\mathrm{Hom}_\mathbb{Z}(P,\mathbb{C}^\ast)\simeq (\mathbb{C}^\ast)^n$
  and we denote by $\mathbb{S}_Q\simeq (\mathbb{S}^1)^n$ its compact torus. We have an open
   dense immersion  $\mathbb{T}_Q\hookrightarrow X$ with an action of $\mathbb{T}_Q$ on $X$ which extends the action of $\mathbb{T}_Q$
   on its self by translations.\\

The toric varieties have a rich geometry that can be related to the geometry of polytopes.
Many results in algebraic geometry and complex differential geometry can be tested on them, for instance the Riemann-Roch
formula. \\

On toric varieties, some properties of line bundles can be interpreted in terms of convex geometry. Let $D$ be  an equivariant Cartier divisor  on $X$ also called a toric
divisor, that is  a Cartier divisor which is  invariant under the action
of the torus $\mathbb{T}_Q$.
 Let $s_D$
 be the rational section of  $\mathcal{O}(D)$ associated to $D$. We know that $D$ defines  a \textit{$\Sigma$-linear support function} $\Psi_D$ on $\Sigma$ (see
\cite[Definition p. 66]{Oda}) and $D$ is uniquely defined by this function (see \cite[Proposition 2.1 (v)]{Oda}).
Moreover the function $\Psi_D$ defines a convex polytope:
 \[
 \Delta_D:=\{x\in P_\mathbb{R}|\, <x,u>\geq \Psi_D(u),\, \forall \, u\in Q_{\mathbb
 {R}}\},
 \]
 where $P_\mathbb{R}:=P\otimes_\mathbb{Z} \R$.  $\Delta_D$ and $\psi_D$ encode  many geometric
 informations about  $D$, for instance

\begin{enumerate}
\item $H^0(X,\mathcal{O}(D))=\oplus_{\e\in \Delta_D\cap P}\mathbb{C} \chi^{\e}$, where $\chi^{\e}$ denotes the character
associated to $\e$, see \cite[Lemma 2.3]{Oda}.
\item $\mathrm{vol}(\mathcal{O}(D))=n! \mathrm{vol}(\Delta_D)$, (this follows from $(1.)$).
\item $D$ is nef if and only if $\Psi_D$ is concave (see \cite[Theorem 2.7]{Oda}).
\item If $D$ is nef then $\deg_D(X)=n!\mathrm{vol}(\Delta_D)$.
\item $D$ is big if and only if $\dim (\Delta_D)=n$, by $(2.)$.
\item $D$ is ample if and only if $\Psi_D$ is strictly concave (see \cite[Theorem 2.13]{Oda}).
\end{enumerate}

  If we denote by $X_{\geq }$ the quotient
of $X$ by $\mathbb{S}_Q$. Then the open subset $X_{\geq}^\circ$ can be identified with $\mathrm{Hom}_\mathbb{Z}(P,\R_{>0})$
(see \cite[\S 4]{Fulton}). But, since $Q_\mathbb{R}=\mathrm{Hom}_\mathbb{Z}(P,\R)$ then the usual exponential defines a morphism
$\exp(-(\cdot)):Q_\mathbb{R}\rightarrow X_{\geq}^\circ\hookrightarrow X(\C)$. This morphism
extends in a obvious way to $Q_\C\rightarrow \mathbb{T}_Q\hookrightarrow X(\C)$.\\

We say that a function $f$ on $X$ (resp. a hermitian metric $\vc$ on $\mc$) is $\mathbb{S}_Q$-invariant if
$f(z\cdot x)=f(x)$ (resp. $\|s_D\|(z\cdot x)=\|s_D\|(x)$) for any $x\in X$ and $z\in \mathbb{S}_Q$. A continuous
hermitian metric $\vc$ on $\mc$ is said semipositive if the Chern current  $c_1(\mc,\vc)$ is nonnegative. An
admissible metric is by definition a uniform limit of a sequence of smooth hermitian semipositive metrics. When
$\mathcal{O}(D)$ is ample  then the notions of admissible metrics and semipositives metrics are equivalent, see
\cite[Theorem 4.6.1]{Maillot}.\\

The function $\Psi_D$ defines a continuous hermitian metric on $\mc$  called the \textit{canonical metric}
 of $\mc$. It is
given locally as follows:  The norm of a local holomorphic section $s$ at a point $x$ is the following real
\[
\|s(x)\|_{\infty,D}=\Bigl|\frac{s}{\chi^{\Psi_D}}(x) \Bigr|.
\]
(see \cite[Proposition 3.3.1]{Maillot}). When $\mc$ is spanned by its global sections (equivalently $\Psi_D$ is
concave) one can show that  $\|\cdot\|_{D,\infty}=\phi_D^\ast \|\cdot\|_{ \infty}$ where
$\phi_D:X\rightarrow \p^{\#(\Delta_D\cap P)-1}$ is the equivariant morphism defined in terms of $\Delta_D\cap P$ and
$\|\cdot\|_\infty$ is the canonical metric of $\mathcal{O}(1)$ (see \cite[\S 3.3.3]{Maillot} for a detailed
construction). Moreover we have that $\|\cdot\|_{\infty,D}$ is a uniform limit of
a sequence of  smooth semipositive metrics and  $-\log \|s_D\|_{\infty,D}^2$ is a plurisubharmonic  weight on $\mathcal{O}(D)$
\cite[Propositions 3.3.11, 3.3.12]{Maillot}.\\

Let $\|\cdot\|_{\overline{D}}$ be a $\mathbb{S}_Q$-invariant hermitian metric on $\mc$ such that $\|\cdot\|_{\overline{D}}/\|\cdot\|_{\infty,D}$
 is bounded on $X$. We let $\overline{D}:=(D,\|\cdot\|_{\overline{D}})$ the obtained  hermitian line
 bundle and we called it a \textit{toric metrized divisor}. We set $g_{\overline{D}}:Q_\mathbb{R}\rightarrow \R$ the function
  defined as follows:
 \[
 g_{\overline{D}}(u):=\log \|s_D(\exp(-u))  \|_{\overline{D}}\quad \forall u\in Q_\mathbb{R}.
 \]
By definition of the canonical metric one can check easily that
\begin{equation}\label{inf}
g_{\overline{D}_\infty}(u)=\inf\{<v, u>| \, v\in \Delta_D\}\quad \forall\, u\in Q_{\R}.
\end{equation}
($<v,u>$ denotes the pairing defined by $Q_\mathbb{R}$ and $P_\mathbb{R}$).
We denote by  $\check{g}_{\overline{D}}:P_\mathbb{R}\rightarrow [-\infty,+\infty[$ the Legendre-Fenchel  transform of $g_{\overline{D}}$, i.e the function defined
 for any $x\in P_\mathbb{R}$ as follows
\[
\check{g}_{\overline{D}}(x):=\inf_{x\in Q_\mathbb{R}}(<x, u>-g_{\overline{D}}(u)).
\]
 We have $\check{g}_{\overline{D}_\infty}$ vanishes on $\Delta_D$ and equal to $-\infty$ otherwise
(One can show that this follows from the following assertion: $\check{g}_{\overline{D}_\infty}(x)=t \check{g}_{\overline{D}_\infty}(x)$ for any $x\in \Delta_D$ and $t>0$, which is an easy consequence of \eqref{inf}).
   Combining this with  Proposition \ref{x3}, we can show that $\check{g}_{\overline{D}}(x)$ is finite if and only if  $x\in \Delta_D$ and   $\check{g}_{\overline{D}}$ is   concave on $\Delta_D$.\\

Let $\|\cdot\|_{\overline{D}_0}$ and $\|\cdot\|_{\overline{D}_1}$ be two smooth  hermitian  metrics on $D$, $\phi_{\overline{D}_0}$ and $\phi_{\overline{D}_1}$
 the associated
weights. We define the Monge-Amp\`ere
functional $\mathcal{E}$ by the formula

\[
\mathcal{E}(\overline{D}_1)-\mathcal{E}(\overline{D}_0)=\frac{1}{n+1}\sum_{j=0}^n\int_X
-\log\frac{\|\cdot\|_{\overline{D}_1}}{\|\cdot\|_{\overline{D}_0}}c_1(\overline{D}_0)^{\wedge j}\wedge c_1( \overline{D}_1)^{\wedge n-j}.
\]

By the theory of Bedford-Taylor \cite{Bedford}, this definition extends to admissible metrics,  and hence to
integrable ones by polarisation. By definition an integrable metric  can be written, in additive notation,
 as a difference of two admissible metrics. \\

 Following \cite{BermanBoucksom},   when $\mathcal{O}(D)$ is big we set
 \[
  \mathcal{E}_{\rm{eq}}(\overline{D}_1)-\mathcal{E}_{\rm{eq}}
(\overline{D}_0):=\frac{1}{\mathrm{Vol}(D)}(\mathcal{E}((\overline{D}_1)_X)-\mathcal{E}((\overline{D}_0)_X).
 \]
 where $(\overline{D}_i)_X$ is the metrized toric divisor $D$ endowed with the weight $P_X\phi_{\overline{D}_i}$, the equilibrium  weight of $\phi_{\overline{D}_i}$ for $i=0,1$.
In \cite[\S 1.3]{BermanBoucksom}, $\mathcal{E}_{\rm{eq}}(\overline{D})$ is called the energy at equilibrium of $(X,\phi_{\overline{D}})$
($\phi_{\overline{D}}$ is the weight of $\overline{D}$). \\

   Our first result is Theorem \ref{x4} which gives  an integral representation of
the variation of the energy functional $\mathcal{E}$ in terms of some
combinatorial objects defined on the polytope associated to $D$. This theorem can be seen as a toric version of \cite[Theorem B]{BermanBoucksom}.\\

Let $\mu$ be  a probability measure with non-pluripolar support on $X$.
We endow the space of global sections $H^0(X,\mathcal{O}(D))$ with the $L^2$-norm
\[
\|s\|_{L^2(\mu, \overline{D})}^2:=\int_X \|s\|_{\overline{D}}^2\mu.
\]
Also we consider the sup norm defined as follows
\[
\|s\|_{\sup,  \overline{D}}:=\sup_{x\in X}\|s\|_{\overline{D}}(x).
\]
for any $s\in H^0(X,\mathcal{O}(D))$. Let $k\in \N^\ast$.
We consider  the following functional
\[
\mathcal{L}_k(\mu,k\overline{D}):=\frac{1}{2k N_k}\log \mathrm{vol}_k B^2(\mu, k\overline{D}),
\]
 where $\mathrm{vol}_k B^2(\mu, k\overline{D})$ is by definition the volume of
the unit ball $B^2(\mu, k\overline{D})$ in $H^0(X,\mathcal{O}(kD))$ with respect to the $L^2$-norm, and $N_k:=\dim H^0(X,\mathcal{O}(kD))$.\\

The Bergman distortion function $\rho(\mu,\overline{D})$ is by definition the function given at a point $x\in X$ by
\[
\rho(\mu,\overline{D})(x)=\sup_{s\in H^0(X,\mathcal{O}(D))\setminus\{0\}}\frac{\|s(x)\|_{\overline{D}}^2}{\quad\|s\|_{L^2(\mu, \overline{D})}^2}.
\]
If $\{s_1,\ldots,s_N\}$ is a $L^2(\mu, \overline{D})$-orthonormal basis of $H^0(X,\mathcal{O}(D)$, then it is
 well known that
 \[
 \rho(\mu,\overline{D})(x)=\sum_{j=1}^N \|s_j(x)\|_{\overline{D}}^2\quad \forall\,x\in X.
 \]

 \begin{definition}\label{bm}
 We say that  $\mu$ has the Bernstein-Markov property with respect to  $\|\cdot\|_{\overline{D}}$ if $\forall \eps>0$ we have
 \[
 \sup_X \rho(\mu,k\overline{D})^\frac{1}{2}=O(e^{k\eps}).
 \]
 \end{definition}

 If $\mu$ is a smooth positive volume form and $\|\cdot\|_{\overline{D}}$ is a continuous metric on $\mathcal{O}(D)$ then
 $\mu$ has the Bernstein-Markov property with respect to $\|\cdot\|_{\overline{D}}$  (see \cite[Lemma 3.2]{BermanBoucksom}).\\

 Our main result is the following theorem:

\begin{theorem}\label{g1}[Main theorem]
Let $X$ be a complex projective  nonsingular toric  variety and  $D$ a toric  divisor on $X$ such that
$\mathcal{O}(D)$ is big and nef. Let  $\overline{D}_i:=(D,\|\cdot\|_{\overline{D}_i})$ be a  continuous toric metrized divisor on $X$ for
$i=0,1$. Let
$\mu_j$ be a probability measure which is  $\mathbb{S}_Q$-invariant on $X$  and
 with the Bernstein-Markov property with respect to
$\|\cdot\|_{\overline{D}_j}$, $j=0,1$. Then as $k\rightarrow \infty$ we have
\[
\mathcal{L}_k( \mu_1, \overline{D}_1)-\mathcal{L}_k(\mu_0,\overline{D}_0)\rightarrow \mathcal{E}_{\rm{eq}}(\overline{D}_1)-\mathcal{E}_{\rm{eq}}
(\overline{D}_0).
\]
\end{theorem}
This theorem describes the asymptotics of $\mathcal{L}_k( \mu, \overline{D})$, the  functional  volume of the balls of the holomorphic
sections of
a continuous toric divisor $\overline{D}$ when $k$ tends to $\infty$. In particular we recover partially
a result of Berman and Boucksom
\cite[Theorem A]{BermanBoucksom}.  Comparing to \cite{BermanBoucksom} our approach is  completely different.
In fact, our strategy  is  based mainly on the  combinatorial structure of the toric variety,
which makes the proof much easier. A crucial ingredient in the proof of  Theorem \ref{g1} is Theorem
\ref{x4}.

\section{The Monge-Amp\`ere operator and the equilibrium weight}

We keep the same notations as in the introduction. 
Let $X$ be a complex projective nonsingular toric variety and and $L$  a holomorphic line bundle over  $X$.
Let $\phi$ be a weight
of a continuous hermitian metric $e^{-\phi}$
on $L$. When $\phi$ is smooth we define the Monge-Amp\`ere operator as
\[
\mathrm{MA}(\phi):=(dd^c\phi)^{\wedge n}.
\]
The equilibrium weight of $\phi$ is defined as:
\[
P_X\phi:=\sup\{\psi|\,\psi\, \text{psh weight on}\, L,\, \psi\leq \phi\,\text{on}\, X \}.
\]
It is  known that the equilibrium weight is upper semicontinuous psh weight
  with
minimal singularities.

\begin{proposition}\label{w2}
Let $\phi_1$ and $\phi_0$ be two continuous weights on $L$ a big line bundle on $X$. We have
\[
|P_X\phi_1-P_X\phi_0|\leq \sup_{x\in X}|\phi_1-\phi_0|.
\]

\end{proposition}
\begin{proof}
This an easy consequence of the definition of the equilibrium weight.
\end{proof}

Assume that $L$ is ample and let $\m$ be a positive $(1,1)$-form in $c_1(L)$. We set
\[
\mathcal{H}_\m:=\{u\in \cl(X)|\, dd^cu+\m>0  \}.
\]
Clearly, $\mathcal{H}_\m$ is convex subset which is identified with the set of smooth  positive hermitian metrics
(resp. weights) on $L$. We set
$P_{\omega}(u):=\sup\{v\in \mathcal{H}_\m| v\leq u\}$. Following \cite{BermanBoucksom}, \cite{BBGZ}, this
operator  extends
to $\mathcal{C}^0(X)$  with image in $\mathcal{C}^0(X)\cap \overline{\mathcal{H}_\m}$. In other words,
if $\phi$ is continuous weight on $L$ an ample line bundle then $P_X\phi$ is also continuous.

\begin{remarque}\label{rem}
\rm{Let $D$ be  a toric nef divisor on $X$. Let $\|\cdot\|_{\overline{D}}$
(resp. $\phi_D$) a continuous hermitian metric  (resp. continuous weight) on $D$. Then
$\|\cdot\|_{\overline{D}_X}/\|\cdot\|_{\infty,D}$ (resp. $P_X\phi_D-\phi_{\infty,D}$) is bounded on $X$, where $\overline{D}_X$ is
hermitian line bundle $\mc$ endowed with the metric defined by $P_X\phi_D$. Indeed, since $D$ is nef then
 $\|\cdot\|_{\infty,D}$ is a semipositive metric. Then we can find  a constant $C$
such that $\phi_{\infty,D}-C\leq P_X\phi \leq \phi_D$.}
\end{remarque}

  \begin{proposition}\label{x3}
Let $g$ be a real function on $Q_\mathbb{R}$. Then $g$ defines a hermitian (continuous) metric $\|\cdot\|_g$ on
$\mathcal{O}(D)$ if and
only if
$g-\Psi_D$ extends to a bounded (continuous) function on $X_{\geq }$. Moreover, we have
\[
\sup_{x\in \Delta_D}|\check{g}-\check{g'}|\leq \sup_{u\in Q_\mathbb{R}}|g-g'|,
\]
for any $g$ and $g'$ two functions on $Q_\mathbb{R}$ defining hermitian metrics on $\mathcal{O}(D)$.
\end{proposition}
\begin{proof}
The proof is an easy consequence of the definitions.
\end{proof}

\begin{lemma}\label{eq} Let $\overline{D}=(D,\|\cdot\|_{\overline{D}})$ be  a continuous  metrized   divisor such that $\mathcal{O}(D)$ is big
and nef on
$X$. We set $\phi_D:=-\log \|s_D\|_{\overline{D}}$ and we denote by $\overline{D}_X$ the metrized toric divisor $D$ endowed
with $P_X\phi_D$. Then $P_X\phi_D$ is a $\mathbb{S}_Q$-invariant weight on $D$  and the
following equality holds on $\Delta_D$
\[
\check{g}_{\overline{D}}=\check{g}_{{\overline{D}}_X}.
\]

\end{lemma}
\begin{proof} By definition $\phi_D$ is $\mathbb{S}_Q$-invariant weight. Let $\psi$  be a psh weight on $\mathcal{O}(D)$
such that $\psi\leq \phi$. For any $t\in \mathbb{S}_Q$, we set $\psi_t:=\psi(t\cdot (\cdot))$. Then $\psi_t$
 is a psh weight
verifying $\psi_t(z)\leq \phi_D(t\cdot z)=\phi_D(z)$ for any $z\in X$. That is $\psi_t\leq \phi_D$. We conclude that $(P_X\phi_D)_t
\leq P_X\phi_D$ for any $t\in  \mathbb{S}_Q$. It follows that $P_X\phi_D$ is a
$\mathbb{S}_Q$-invariant weight on $D$. By \eqref{rem}, $\check{g}_{{\overline{D}}_X}$ is well defined. Moreover,
we have $\|s\|_{\sup, k\overline{D}}=\|s\|_{\sup, k\overline{D}_X}$ for any $k\in \N$ and $s\in H^0(X,\mathcal{O}(kD))$ (see for instance
\cite[Proposition 2.8]{BermanBoucksom}). Let $k\in \N^\ast$ and  $\e\in k\Delta_D\cap P $. We have
$\|\chi^\e\|_{\sup,k\overline{D}}=\sup_{x\in X}\|\chi^\e(x)\|_{k\overline{D}}=\sup_{u\in Q_\mathbb{R}}\|\chi^\e(\exp(-u))\|_{k\overline{D}}=\exp(-k\inf_{u\in
Q_\mathbb{R}}(\frac{\e}{k}\cdot u-g_{\overline{D}}(u))=\exp(-k\check{g}_{\overline{D}}(\frac{\e}{k}))$. We deduce that
\[
\check{g}_{\overline{D}}(x)=\check{g}_{\overline{D}_X}(x)\quad \forall\,x\in \Delta_D\cap \mathbb{Q}^n.
\]
Using
the fact that a concave and finite function on a $\Delta_D$ is necessarily continuous on  its interior, see
\cite[Theorem 10.1]{convex}, then we get $\check{g}_{\overline{D}}=\check{g}_{{\overline{D}}_X}$ on $\mathrm{Int}(\Delta_D)$. But,
since $\Delta_D$ is the convex closure of $\Delta_D\cap P$ then
\[
\check{g}_{\overline{D}}(x)=\check{g}_{\overline{D}_X}(x)\quad \forall\,x\in \Delta_D.
\]

\end{proof}

\section{The energy functional in the toric setting}

The goal of this section is to give a formula for  the variation of the energy functional $\mathcal{E}$ in terms of  the Legendre-Fenchel transform. First, this formula
is proved in the ample case, see Theorem \ref{x4}, then we deduce the general case of big and nef divisors in Corollary \ref{cor}.

\begin{proposition}\label{av} Let $D$ be a toric divisor on $X$.
Assume that there exists  $\vc$  an admissible and $\mathbb{S}_Q$-invariant metric on $\mathcal{O}(D)$. Then there exists a
sequence of smooth, semipositive and $\mathbb{S}_Q$-invariant hermitian metrics converging uniformly to
$\vc$.
\end{proposition}

\begin{proof}
First let recall that given a smooth hermitian metric $\vc$ one can average it in order to get
a $\mathbb{S}_Q$-invariant smooth metric. This is done as follows,  we  define the metric $\|\cdot\|_{\mathbb{S}_Q}$
given on $\mathbb{T}_Q$ by $\log \|s(x)\|_{\mathbb{S}_Q}=\int_{\mathbb{S}_Q}\log \|s(t\cdot x)\|d\mu_{\mathrm{Haar}}$.
This metric extends to $X$ since  $\log(\|s(x)\|_{\mathbb{S}_Q}/\|s(x)\|')= \int_{\mathbb{S}_Q}\log (\|s(t\cdot
x)\|/\|s(x)\|') d\mu_{\mathrm{Haar}}$ where $\vc'$ is a smooth and $\mathbb{S}_Q$-invariant metric, extends to
a smooth function to $X$. Clearly the metric $\|\cdot\|_{\mathbb{S}_Q}$ is $\mathbb{S}_Q$-invariant and smooth.  Moreover
$c_1(\mathcal{O}(D),\|\cdot\|_{\mathbb{S}_Q})=\int_{\mathbb{S}_Q}t^\ast c_1(\mathcal{O}(D),\vc)d\mu_{\mathrm{Haar}}$
where $t^\ast$ is the pull-back defined by the multiplication by $t$. It follows that if
 $c_1(\mathcal{O}(D),\vc)\geq 0$
then $c_1(\mathcal{O}(D),\|\cdot\|_{\mathbb{S}_Q})\geq 0$.

Let $\vc$ be an admissible hermitian metric. By definition there exists $(\|\cdot\|_n)_{n\in \N}$ a sequence
of smooth, semipositive hermitian metrics converging uniformly to $\vc$. By averaging  this sequence as before
we get a sequence of  smooth, semipositive and $\mathbb{S}_Q$-invariant hermitian metrics which converges uniformly to
$\vc$.

\end{proof}

Let $\overline{D}$ be a smooth positive toric divisor on $X$. We set $\Psi_D=-\log \|s_D\|^2_{\overline{D}}$.
 The exponential map gives the following change of variables, $z=\exp(-u-i\theta)\in \mathbb{T}_Q$
for any $z\in \mathbb{T}_Q$  where $u, \theta\in Q_\mathbb{R}$. Then  $\frac{\pt^2\Psi_D}{\pt z_k\pt\z_l}=\frac{1}{z_k \z_l}\frac{\pt^2g_{\overline{D}}}{\pt u_k\pt u_l}$ for
$ k,l=1,\ldots,n$.
Since $\overline{D}$ is positive and smooth then $g_{\overline{D}}$ is  a strictly concave smooth function on $Q_\mathbb{R}$.
Hence, for any $x\in \mathrm{Int}(\Delta_D)$ there exists a unique $G_{\overline{D}}(x)\in Q_\mathbb{R}$ such that
$\check{g}(x)=x\cdot G_{\overline{D}}(x)-g_{\overline{D}}(G_{\overline{D}}(x))$, and we can show that $G_{\overline{D}}$
is smooth on $\Delta_D$  and
$\frac{\pt g_{\overline{D}}}{\pt u}\circ G_{\overline{D}}=\mathrm{Id}_{\Delta_D}$ (this follows from
\cite[Theorem 26.5]{convex}). In other words, $x:=\frac{\pt g_{\overline{D}}}{\pt u}$ defines a
 $\cl$-diffeomorphism between $\mathrm{Int}(\Delta_D)$ and $Q_\mathbb{R}$, and we have
 \[
 c_1(\overline{D})^{\wedge n}=\frac{n!}{(2\pi)^n}
 \det\bigl(\frac{\pt^2g_{\overline{D}}}{\pt u_k\pt u_l}\bigr)_{{}_{1\leq k,l\leq n}}du_1\wedge \cdots\wedge du_n\wedge
 d\theta_1\wedge \cdots\wedge d\theta_n=\frac{n!}{(2\pi)^n}dx_1\wedge \cdots dx_n\wedge d\theta_1\wedge \cdots d\theta_n.
 \]
 Let $\phi$ be a continuous function on $X$ which is
 invariant under the action of $\mathbb{\mathbb{S}_Q}$. We denote by $\phi_D$ the function on $\Delta_D$ given by
  $\phi_D(x)=\phi(\exp(-G_{\overline{D}}(x)))$. One can show that
  \[
  \int_X \phi c_1(\overline{D})^{\wedge n}=n!\int_{\Delta_D}\phi_D dx,
  \]
where $dx=dx_1\wedge \cdots dx_n$ denotes the standard Lebesgue measure on $Q_\mathbb{R}$. In particular, one have the following
identity
 $\deg_D(X)=n! \mathrm{vol}(\Delta_D)$, which extends easily to nef divisors.\\

 \begin{lemma}\label{f1}
 Let $f$ be a smooth function on $X$ and $\mathbb{S}_Q$-invariant. We have
 \[
 \int_X df\wedge  d^c f\wedge c_1(\overline{D})^{n-1}=\int_{\Delta_D}< \frac{dG_{\overline{D}}(x)}{dx}\cdot
 \frac{\pt f}{\pt x}, \frac{\pt f}{\pt x}> dx,
 \]
 where $f(x):=f(\exp(-G_{\overline{D}}(x)))$ on $\Delta_D$.
 \end{lemma}
 \begin{proof}
Notice that $\frac{dG_{\overline{D}}(x)}{dx}=\mathrm{Hess}(\check{g})$ and  using the change of coordinates
 $x:=\frac{dg_t}{dt}$ one can deduce the lemma.
 \end{proof}

Let $\|\cdot\|_{\overline{D}_0}$ and $\|\cdot\|_{\overline{D}_0}$ be two smooth, $\mathbb{S}_Q$-invariant and positive
hermitian
metrics on $\mc$ and we let  $\overline{D}_i:=(D,\|\cdot\|_{\overline{D}_i})$ for $i=0,1$.  For any $t\in [0,1]$ we set $g_t=tg_1+(1-t)g_0$.
Then we have the following result

\begin{proposition}\label{x1}
 The following function defined on $[0,1]$
\[
t\mapsto \int_{\Delta_D}\check{g}_tdx,
\]
is differentiable on $[0,1[$. Moreover, we have
\[
\frac{d}{dt}\Bigl(\int_{\Delta_D}\check{g}_tdx\Bigr)_{|_{t=0^+}}=-\int_{\Delta_D}
(g_1-g_0)(G_0(x))dx.
\]
\end{proposition}

\begin{proof} We denote by $\overline{D}_t$ the positive and smooth toric metrized divisor $D$ endowed with $g_t$ and
we set $G_t:=G_{\overline{D}_t}$ for any $t\in [0,1]$.
We have
\[
\check{g}_t(x)=<x, G_t(x)>-g_0(G_t(x))-t(g_1-g_0)(G_t(x))\quad \forall\,x\in \Delta_D,\forall t\in [0,1].
\]
We set $\mathcal{F}_t(x):=<x, G_t(x)>-g_0(G_t(x))$ for any $x\in \Delta_D$ and $t\in [0,1]$. We have
$\frac{d\mathcal{F}_t}{dt}(x)=<x, \frac{d G_t(x)}{dt}>-\frac{dg_0}{du}_{|_{u=G_t(x)}}\cdot
\frac{d G_t(x)}{dt}=<x-G_0^{-1}(G_t(x)), \frac{dG_t(x)}{dt}>$. That is
\begin{equation}\label{w1}
\frac{d\mathcal{F}_t}{dt}(x)=<x-G_0^{-1}(G_t(x)), \frac{dG_t(x)}{dt}> \quad \forall\,x\in \Delta_D,
\forall t\in
[0,1].
\end{equation}
Let $u\in Q_\mathbb{R}$ such that $x=\frac{dg_t}{dt}(u)$. Then $x-G_0^{-1}(G_t(x))=\frac{dg_t}{dv}(u)-\frac{dg_0}{dv}(u)=
t\bigl(\frac{dg_1}{dt}(u)-\frac{dg_0}{dt}(u) \bigr)$. Recall that $G_t(\frac{dg_t(v)}{dt})=v$ for any $v\in
Q_\mathbb{R}$. This gives  $\frac{dG_t}{dt}(\frac{dg_t}{dv})=-\frac{dG_t}{dx}(\frac{dg_t}{dv})\cdot \frac{\pt (g_1-g_0)}{\pt v}$.
Then \eqref{w1} becomes
\[
\frac{d\mathcal{F}_t}{dt}(x)=-t<
\frac{dG_t}{dx}(\frac{dg_t(u)}{dv})\cdot \bigl(\frac{dg_1}{dv}(u)-\frac{dg_0}{dv}(u)\bigr), \frac{dg_1}{dv}(u)-\frac{dg_0}{dv}(u)  >.
\]
Which is equivalent to
\[
\frac{d\mathcal{F}_t}{dt}(x)=-t<
\frac{dG_t}{dx}(x)\cdot \bigl(\frac{dg_1}{dv}(G_t(x))-\frac{dg_0}{dv}(G_t(x)),
\frac{dg_1}{dv}(G_t(x))-\frac{dg_0}{dv}(G_t(x)) >.
\]

The function $f:=\log \frac{\|\cdot\|_{\overline{D}_1}}{\|\cdot\|_{\overline{D}_0}}$  is smooth and $\mathbb{S}_Q$-invariant on $X$. Then by \eqref{f1}
we have
\[<
\frac{dG_t}{dx}(x)\cdot \bigl(\frac{dg_1}{dv}(G_t(x))-\frac{dg_0}{dv}(G_t(x)),
\frac{dg_1}{dv}(G_t(x))-\frac{dg_0}{dv}(G_t(x)) >dx\wedge d\theta
=df\wedge d^c f\wedge c_1(\overline{D}_t)^{n-1}\] which is absolutely integrable. Therefore,
\[
\frac{d}{dt}\int_{\Delta_D}\mathcal{F}_tdx =\int_{\Delta_D}\frac{d \mathcal{F}_t}{dt}dx=-t
\int_X
d(\log \frac{\|\cdot\|_1}{\|\cdot\|_0})\wedge d^c (\log \frac{\|\cdot\|_1}{\|\cdot\|_0})\wedge c_1(\overline{D}_t)^{n-1}\quad \forall\, t\in [0,1[.
\]

With similar arguments we can establish that $\int_{\Delta_D} (g_1-g_0)(G_t(x))dx$ is also differentiable on $[0,1[$. We conclude that
\[
t\mapsto \int_{\Delta_D}\check{g}_tdx,
\]
is differentiable, and we have
\[
\frac{d}{dt}\Bigl(\int_{\Delta_D}\check{g}_tdx\Bigr)_{|_{t=0^+}}=-\int_{\Delta_D}
(g_1-g_0)(G_0(x))dx.
\]

\end{proof}

\begin{theorem}\label{x4}Let $\mathcal{O}(D)$ be an  ample line bundle on $X$.
 Let $\|\cdot\|_{\overline{D}_1}$ and $\|\cdot\|_{\overline{D}_0}$ be two smooth, $\mathbb{S}_Q$-invariant and positive
hermitian
metrics on $\mathcal{O}(D)$. We have,
\[
\mathcal{E}(\overline{D}_1)-\mathcal{E}(\overline{D}_0)=-\int_{\Delta_D}(\check{g}_{\overline{D}_1}(x)-\check{g}_{\overline{D}_0}(x))dx.
\]
Moreover, this equality extends to admissible metrics.
\end{theorem}

\begin{proof}
 Let $\|\cdot\|_1$ and $\|\cdot\|_2$ be
 two smooth, positive  and $\mathbb{S}_Q$-invariant hermitian metrics on $\mathcal{O}(D)$. Let $s\in [0,1]$ and  let
 $\|\cdot\|_s$ be the metric defined by $g_s:=(1-s)g_{\overline{D}_0}+sg_{\overline{D}_1}$. This metric  is smooth, positive and $\mathbb{S}_Q$-invariant.
 We denote by $\overline{D}_s$ the metrized toric divisor $D$ endowed with the metric $\|\cdot\|_s$.
 Applying   Proposition \ref{x1} one get  for any $s\in [0,1[$
\[
\frac{d}{dt}\Bigl(\int_{\Delta_D}\check{g}_t\Bigr)_{|_{t=s^+}}=-\int_{\Delta_D}(g_1-g_0)(G_s(x))dx.
\]
From the definition of the Monge-Amp\`ere functional, one get easily
\[
\frac{d}{dt}\bigl(\mathcal{E}(\overline{D}_s)-\mathcal{E}(\overline{D}_0)
\bigr)_{|_{t=s^+}}=\int_X(\log\frac{\|\cdot\|_0}{\|\cdot\|_1})c_1(L,\|\cdot\|_s)^{\wedge n}.
\]
 Now, by using the change of variables
$u=G_s(x)$. We get $\int_X(\log\frac{\|\cdot\|_0}{\|\cdot\|_1})c_1(L,\|\cdot\|_s)^n=-\int_{\Delta_D}(g_1-g_0)(G_s(x))dx$.
Thus,
\[
\frac{d}{dt}\bigl(\mathcal{E}(\overline{D}_s)-\mathcal{E}(\overline{D}_0)
\bigr)_{|_{t=s^+}}=\frac{d}{dt}\Bigl(\int_{\Delta_D}\check{g}_t\Bigr)_{|_{t=s^+}}\quad \forall s\in [0,1[.
\]
Therefore
\begin{equation}\label{x2}
\mathcal{E}(\overline{D}_1)-\mathcal{E}(\overline{D}_0)=-\int_{\Delta_D}(\check{g}_{\overline{D}_1}(x)-\check{g}_{\overline{D}_0}(x))dx.
\end{equation}
Suppose now that $\|\cdot\|_0$ and $\|\cdot\|_1$ are admissible. By definition, there exists $(\|\cdot\|_{i,n})_{n\in \N}$ a sequence
of smooth and semipositive metrics on $L$ converging uniformly to $\|\cdot\|_i$, $i=1,2$. By Proposition  \ref{av}, we can assume that
the metrics of the sequences are $\mathbb{S}_Q$-invariant. Moreover, we can  also suppose that
$\|\cdot\|_{i,n}$ is positive for all $n\in \N$ and $i=0,1$. Indeed, let $\vc'$ be a smooth, positive and
$\mathbb{S}_Q$-invariant metric
on $\mathcal{O}(D)$ (for example one can take the pull-back of the Fubini-Study  by an equivariant morphism defined
by  $D$) then $\|\cdot\|_{i,n}^{1-1/n}{\vc'}^{1/n}$ is positive, smooth and $\mathbb{S}_Q$-invariant. We have \eqref{x2} holds
for $g_i=g_{n,i}$ for any $n\in \N$ and $i=0,1$. By the theory of Beford-Taylor the LHS
 converges to $\mathcal{E}(\overline{D}_1)-\mathcal{E}(\overline{D}_0)$. By \eqref{x3} the RHS converges to
 $\int_{\Delta_D}(\check{g}_{\overline{D}_1}(x)-\check{g}_{\overline{D}_0}(x))dx$.
\end{proof}

\begin{proposition}\label{x5}
Let $\overline{D}$ be a   toric metrized divisor and $\overline{A}$ a  toric metrized divisor such that
$A$ is effective.
We have
\[
\lim_{l\in \N, l\rightarrow \infty}\frac{1}{l^{n+1}}\int_{l\Delta_D+\Delta_A}(l
g_{\overline{D}}+g_{\overline{A}})\check{}dx
=\int_{\Delta_D} \check{g}_{\overline{D}}dx.
\]
($dx$ denotes the standard Lebesgue measure on $Q_\mathbb{R}$).
\end{proposition}

\begin{proof}

We set $g_{l\overline{D}+\overline{A}}:=lg_{\overline{D}}+ g_{\overline{A}} $  and
$g_{\overline{D}+\frac{1}{l}\overline{A}}:=g_{\overline{D}}+\frac{1}{l} g_{\overline{A}} $ for any
$l\in \N^\ast$.

The assumption that  $\overline{D}$ and $\overline{A}$ are toric metrized divisors implies that $\sup_{v\in Q_\mathbb{R}}
|g_{\overline{D}}(v)-g_{\overline{D}_\infty}(v)|$ and $
\sup_{v\in Q_\mathbb{R}}|g_{\overline{A}}(v)-g_{\overline{A}_\infty}(v)|$ are  finite.
 There exists a constant $C$ such that $|
\check{g}_{\overline{D}+\frac{1}{l}\overline{A}}|\leq C$ on $\Delta_{D}+\frac{1}{l}\Delta_A$ for any $l\in \N$.
This  follows from the following inequality
\[
| \check{g}_{\overline{D}+\frac{1}{l}\overline{A}}|\leq \sup_{v\in Q_\mathbb{R}}|g_{\overline{D}}(v)-g_{\overline{D}_\infty}(v)|
+
\frac{1}{l}\sup_{v\in Q_\mathbb{R}}|g_{\overline{A}}(v)-g_{\overline{A}_\infty}(v)|,
\]
on $\Delta_D+\frac{1}{l}\Delta_A$ which is a consequence of Proposition \ref{x3} combined with the fact that that $\check{g}_{\overline{D}_\infty
+\frac{1}{l}\overline{A}_\infty}=0$ on $\Delta_D+\frac{1}{l}\Delta_A$. Notice that  $0\in \Delta_A$ because $A$ is effective (this follows from \cite[Proposition 2.1 (v)]{Oda}).
 Then $\check{g}_{\overline{A}}(0)$
is finite and it follows that $g_{\overline{A}}$ is bounded from above. If we multiply the metric of $\overline{A}$
by a positive constant, then it is possible to assume that $g_{\overline{A}}\leq 0$. Observe that the assertion
of the proposition remains true.

By an obvious change of variables, we have
\[
\frac{1}{l^{n+1}}\int_{l\Delta_D+\Delta_A}
\check{g}_{l\overline{D}+\overline{A}}dx=\int_{\Delta_D+\frac{1}{l}\Delta_A}\check{g}_{
\overline{D}+\frac{1}{l}\overline{A}}dx=\int_{\Delta_D}\check{g}_{
\overline{D}+\frac{1}{l}\overline{A}}dx+\int_{(\Delta_D+\frac{1}{l}\Delta_A)\setminus
\Delta_D}\check{g}_{
\overline{D}+\frac{1}{l}\overline{A}}dx.
\]

Fix $x\in \Delta_D$. Since $0\in \Delta_A$ and $g_{\overline{A}}\leq 0$, then $\check{g}_{\overline{D}+\frac{1}{l}\overline{A}}(x)\geq
\check{g}_{\overline{D}+\frac{1}{l'}\overline{A}}(x)\geq \check{g}_{\overline{D}}(x)$ for any $l\leq l'$ in $\N^\ast$.
Let $u\in Q_\mathbb{R}$ such that $\check{g}_{l\overline{D}}(x)=x\cdot u-g_{\overline{D}}(u)$. Then
\[
\check{g}_{\overline{D}+\frac{1}{l}\overline{A}}(x)\leq \check{g}_{\overline{D}}(x)-\frac{1}{l}g_{\overline{A}}(u)
\quad \forall\, l\in \N^\ast.
\]
Therefore $\check{g}_{\overline{D}+\frac{1}{l}\overline{A}}$  is a decreasing function
converging pointwise to $\check{g}_{\overline{D}}$ on $\Delta_D$. Since $0\in \Delta_A$ and $|
\check{g}_{\overline{D}+\frac{1}{l}\overline{A}}|\leq C$ on $\Delta_D$ we conclude (by using the Fatou-Lebesgue theorem)
that\[
\lim_{l\rightarrow \infty}\int_{\Delta_D}\check{g}_{\overline{D}+\frac{1}{l}\overline{A}}=\int_{\Delta_D}\check{g}_{\overline{D}}.
\]
On other hand, we have
\[
\bigl|\int_{(\Delta_D+\frac{1}{l}\Delta_A)\setminus
\Delta_D}\check{g}_{\overline{D}+\frac{1}{l}\overline{A}}\bigr|\leq C\mathrm{vol}((\Delta_D+\frac{1}{l}\Delta_A)\setminus
\Delta_D).
\]

Therefore
\[
\lim_{l\rightarrow \infty}\int_{\Delta_D+\frac{1}{l}\Delta_A}(
g_{\overline{D}}+\frac{1}{l}g_{\overline{A}})\check{}=\int_{\Delta_D}\check{g}_{\overline{D}}.
\]

\end{proof}

\begin{Corollaire}\label{cor}
Let $\mathcal{O}(D)$ be a  big and nef line bundle on $X$.
 Let $\|\cdot\|_{\overline{D}_1}$ and $\|\cdot\|_{\overline{D}_0}$ be two admissible and  $\mathbb{S}_Q$-invariant
hermitian
metrics on $\mathcal{O}(D)$ and set $\overline{D}_i=(D,\|\cdot\|_{\overline{D}_i})$ for $i=0,1$. We have,
\[
\mathcal{E}(\overline{D}_1)-\mathcal{E}(\overline{D}_0)=-\int_{\Delta_D}(\check{g}_{\overline{D}_1}(x)-\check{g}_{\overline{D}_0}(x))dx.
\]
\end{Corollaire}
\begin{proof} Let $\overline{A}$ be a positive and smooth toric metrized divisor. We have $\overline{A}+l\overline{D}_i$ is
a positive continous  toric metrized divisor, for $i=0,1$ and  any $l\in \N$.
Then by Theorem \ref{x4},
\[
\mathcal{E}(\overline{A}+l\overline{D}_1)-\mathcal{E}(\overline{A}+l\overline{D}_0)=-\int_{\Delta_{A+lD}}(\check{g}_{\overline{A}
+l\overline{D}_1}(x)-\check{g}_{\overline{A}
+l\overline{D}_0}(x))dx\quad \forall l\in \N.
\]
We have $\mathcal{E}(\overline{A}+l\overline{D}_1)-\mathcal{E}(\overline{A}+l\overline{D}_0)=l^{n+1}
\bigl(\mathcal{E}(\overline{D}_1)-\mathcal{E}(\overline{D}_0)\bigr)+O(l^{n}),  \forall l\in \N$. Since $A$ is effective and by Proposition  \ref{x5}
 we conclude that

 \[
\mathcal{E}(\overline{D}_1)-\mathcal{E}(\overline{D}_0)=-\int_{\Delta_D}(\check{g}_{\overline{D}_1}(x)-\check{g}_{\overline{D}_0}(x))dx.
\]
\end{proof}

\begin{lemma}\label{f3}
Let $\Theta$ be a convex  compact subset in  $\mathbb{R}^n$ such that $\mathrm{vol}(\Theta)>0$ ($\mathrm{vol}$
denotes the   volume induced by the standard  Lebesgue measure $dx$ of $\mathbb{R}^n$) and let $\vf$ be a
bounded concave function on $\Theta$. We have,
\[
\lim_{l\in \N^\ast, l\mapsto \infty }\frac{1}{l^n}\sum_{\e \in l\Theta\cap P}\vf(\frac{\e}{l})=\int_{\Theta}\vf.
\]
\end{lemma}

\begin{proof} Let $\eps>0$. There exists
$\Theta'$ a convex compact subset in $\mathrm{Int}(\Theta)$ such that $\mathrm{vol}(\Theta\setminus
\Theta')<\eps$. By \cite[Theorem 10.1]{convex}, the concave
function $\vf$ is continuous on $\Theta'$. Then  $\bigl|\lim_{l\in \N^\ast, l\mapsto \infty }\frac{1}{l^n}\sum_{\e \in l\Theta'\cap P}\vf(\frac{\e}{l})-\int_{\Theta'}\vf\bigr|\leq \eps$ for $l\gg 1$. We have
\begin{align*}
\bigl|\frac{1}{l^n}\sum_{\e \in l\Theta\cap P}\vf(\frac{\e}{l})-\int_{\Theta}\vf\bigr|&\leq
\bigl|\frac{1}{l^n}\sum_{\e \in l\Theta'\cap P}\vf(\frac{\e}{l})-\int_{\Theta'}\vf\bigr|+
\bigl|\int_{\Theta}\vf-\int_{\Theta'}\vf  \bigr|+\bigl|\frac{1}{l^n}\sum_{\e \in l(\Theta\setminus \Theta')\cap P}\vf(\frac{\e}{l})\bigr|\\
&\leq \bigl|\frac{1}{l^n}\sum_{\e \in l\Theta'\cap P}\vf(\frac{\e}{l})-\int_{\Theta'}\vf\bigr|+\sup_{\Theta}
|\vf|\mathrm{vol}(\Theta\setminus \Theta')+\sup_{\Theta}|\vf|\bigl(\frac{1}{l^n}\#(l\Theta\cap P)-\frac{1}{l^n}
\#(l\Theta'\cap P) \bigr),
\end{align*}
and since $\lim_{l\mapsto \infty \frac{1}{l^n}\#(l\Theta\cap P)}=\mathrm{vol}(\Theta)$ and $\lim_{l\mapsto \infty \frac{1}{l^n}\#(l\Theta'\cap P)}=\mathrm{vol}(\Theta')$ we conclude that
\[
\frac{1}{l^n}\sum_{\e \in l\Theta\cap P}\vf(\frac{\e}{l})-\int_{\Theta}\vf=O(\eps)\quad \forall l\gg 1.
\]
\end{proof}

\begin{lemma}\label{key}
Let $\Theta$ be a convex  compact subset in  $\mathbb{R}^n$ such that $\mathrm{vol}(\Theta)>0$. For any $l\in \N^\ast$, let $A_l=(a_{\e,\e'})_{\e,\e'\in l\Theta\cap P}$ be a positive symmetric matrix indexed by $l\Theta\cap P$, and let
 $K_l$ be a subset of $\mathbb{R}^{l\Theta\cap P}\simeq \mathbb{R}^{\#(l\Theta\cap P)}$ given by
\[
 K_l=\{(x_\e)\in \mathbb{R}^{l\Theta\cap P}\,|\, \sum_{\e,\e'\in l\Theta\cap P}a_{\e,\e'}x_\e x_{\e'}\leq 1 \}.
\]
We assume  there is an integrable   function  $\vf:\Theta\longrightarrow \R$ such that for any  $\eps>0$, there
 exists a constant  $D$  verifying
\[
 |\log(\frac{1}{a_{\e,\e}})-l\vf(\frac{\e}{l})|\leq D+\eps\, l,
\]
 for any $l\gg 1$ and   $\e\in l\Theta\cap P$. Then, we have

 \[
\lim_{l\mapsto \infty}\frac{1}{l^{n+1}}\sum_{\e\in l\Theta\cap P}\log(\frac{1}{a_{\e,\e}})=\int_\Theta \vf(x)dx.
\]
\end{lemma}
\begin{proof}
 Let $\eps>0$. By assumption, there exists
 constant $D$ such that
\[
\vf(\frac{\e}{l}) -\frac{1}{l}D-\eps \leq \frac{1}{l}\log(\frac{1}{a_{\e,\e}})\leq \vf(\frac{\e}{l})+ \frac{1}{l}D+\eps \quad \forall \, l\gg 1\quad \forall\,\e\in l\Theta\cap P.
\]
Then
\[
\frac{1}{l^d}\sum_{\e\in l\Theta\cap P}\vf(\frac{\e}{l}) -\frac{m_l}{l^{n+1}}D-\frac{m_l}{l^n}\eps \leq \frac{1}{l^{n+1}}\sum_{\e\in l\Theta\cap P}\log(\frac{1}{a_{\e,\e}})\leq\frac{1}{l^n}\sum_{\e\in l\Theta\cap P}\vf(\frac{\e}{l}) +\frac{m_l}{l^{n+1}}D+\frac{m_l}{l^n}\eps  \quad \forall \, l\gg 1.
\]
where $m_l=\#(l\Theta\cap P)$.
Note that
\begin{equation}\label{f2}
\lim_{l\mapsto \infty} \frac{1}{l^n}\sum_{\e\in l\Theta\cap P}\vf(\frac{\e}{l})=\lim_{l\mapsto \infty} \sum_{x\in
\Theta\cap (1/l)P}\vf(x)=\int_\Theta \vf(x)dx.
\end{equation}
Then, we can find $l_0\gg 1$ such that
\[
\bigl| \frac{1}{l^{n+1}}\sum_{\e\in l\Theta\cap P}\log(\frac{1}{a_{\e,\e}})-\int_\Theta \vf(x)dx \bigr|\leq \eps
+\frac{m_l}{l^{n+1}}D+\frac{m_l}{l^n}\eps\quad \, \forall\, l\geq l_0.
\]
Since  $m_l=O(l^n)$. We can deduce that,
\[
\bigl| \frac{1}{l^{n+1}}\sum_{\e\in l\Theta\cap P}\log(\frac{1}{a_{\e,\e}})-\int_\Theta \vf(x)dx \bigr|=
O(\eps)\quad \, \forall\, l\geq l_0.
\]
Then
\[
\lim_{l\mapsto \infty}\frac{1}{l^{n+1}}\sum_{\e\in l\Theta\cap P}\log(\frac{1}{a_{\e,\e}})=\int_\Theta \vf(x)dx.
\]

\end{proof}

\section{{\scshape{The proof of the main theorem}}}

Let $X$ be a complex projective nonsingular toric variety and  $D$ a toric  divisor on $X$ such that
$\mathcal{O}(D)$ is big and nef. Let $\|\cdot\|_{\overline{D}_0}$ and $\|\cdot\|_{\overline{D}_1}$ be two continuous toric metrics on $D$ and
$\overline{D}_i:=(D,\|\cdot\|_{\overline{D}_i})$ and $\phi_{\overline{D}_i}$ the associated weight for $i=0,1$. Let
$\mu_j$ be a probability measure  $\mathbb{S}_Q$-invariant on $X$   with the Bernstein-Markov property with respect to
$\|\cdot\|_{\overline{D}_j}$, $j=0,1$.  We denote by ${\overline{D}_i}_X$ the metrized toric divisor endowed with the weight
$P_X\phi_{\overline{D}_i}$ for $i=0,1$.

\begin{proposition}\label{eq33} We have,
\begin{equation}\label{eq2}
\lim_{k\rightarrow \infty }\mathcal{L}_k(\mu_1,\overline{D}_1)-\mathcal{L}_k(\mu_0,\overline{D}_0)=-
\frac{1}{\mathrm{vol}(D)}\int_{\Delta_D}(\check{g}_{\overline{D}_1}-\check{g}_{\overline{D}_0})dx.
\end{equation}

\end{proposition}
\begin{proof}
For any $\e\in \Delta_D\cap P$, we set $s_{\e}:=\|\chi^{\e}\|_{L^2(\mu_1,\overline{D}_1)}^{-1}\chi^{\e}$ ($\chi^\e$ is
the global section of $\mathcal{O}(D)$ defined by $\e$. Recall that
$H^0(X,\mathcal{O}(D))=\oplus_{m\in \Delta_D\cap P}\C\chi^m$). Since $\mu_0$ and $\mu_1$ are
$\mathbb{S}_Q$-invariant, then $\{s_\e|\, \e\in \Delta_D\cap P \}$ is an
$L^2(\mu_1,\overline{D}_1)$-orthonormal basis of $H^0(X,L)$ and we have
\[
\frac{\mathrm{vol}_{L^2(\mu_1,
\overline{D}_1)}(B(\mu_1,\overline{D}_1))}{\mathrm{vol}_{L^2(\mu_0,
\overline{D}_0)}(B(\mu_0,\overline{D}_0))}=\det\bigl((s_\e,s_{\e'})_{L^2(\mu_0,\overline{D}_0)} \bigr)_{\e,\e'\in \Delta_D\cap P}
=\prod_{\e\in \Delta_D\cap P}(s_{\e},s_{\e}
)_{L^2(\mu_0,\overline{D}_0)}.
\]
Then for any $k\in \N^\ast$,
\[
\mathcal{L}_k(\mu_1,\overline{D}_1)-\mathcal{L}_k(\mu_0,\overline{D}_0)=\frac{1}{k N_k}\log \prod_{\e\in k\Delta_D\cap
P}(s_\e,s_\e
)_{L^2(\mu_0,\overline{D}_0)},
\]
($N_k:=\dim H^0(X,\mathcal{O}(kD))$).
 Since $(\mu_0,\overline{D}_0)$
and $(\mu_1,\overline{D}_1)$ satisfy the Bernstein-Markov property, then for any $\eps>0$ there exists a constant $D$ such that
\[
|\log (s_\e,s_\e)_{L^2(\mu_0,\overline{D}_0)}-k (\check{g}_{\overline{D}_1}-\check{g}_{\overline{D}_0})(\frac{\e}{k})|\leq D+k\eps,  \quad\forall \e
\in k\Delta_D\cap P,\,\forall k\gg 1.
\]
(Notice that we use the fact  $\|s_\e\|_{\sup, k\overline{D}}=\exp(-k\check{g}_{\overline{D}}(\frac{\e}{k}))$, see the proof
of lemma \ref{eq}).
Now let $\Theta:=\Delta_D$ and  $\phi:=\check{g}_{\overline{D}_1}-\check{g}_{\overline{D}_0}$. They satisfy  the assumptions of
 lemma \ref{key} (More precisely, $\phi$ satisfies \eqref{f2} which is a consequence of  lemma \ref{f3}). This
 concludes the proof of the proposition.

\end{proof}

\subsection{{\scshape{The case of ample divisor}}}

Suppose that $D$ is ample.  We know that  $\|\cdot\|_{{\overline{D}_i}_X}$
 is a continuous and  semipositive metric on $D$ for $i=0,1$. By \cite[Theorem 4.6.1]{Maillot}, this metric
is admissible. Using  Corollary \ref{cor}, we deduce
that
\[
\mathcal{E}({\overline{D}_1}_X)-\mathcal{E}({\overline{D}_0}_X)=-\int_{\Delta_D}(\check{g}_{{\overline{D}_1}_X}(x)-\check{g}_{{\overline{D}_0}_X}(x))dx.
\]
Thus,
\[
\mathcal{E}({\overline{D}_1}_X)-\mathcal{E}({\overline{D}_0}_X)=-\int_{\Delta_D}(\check{g}_{{\overline{D}_1}}(x)-\check{g}_{{\overline{D}_0}}(x))dx,
\]
by lemma \ref{eq}.
Now by Proposition  \ref{eq33} we get
\[
\lim_{k\rightarrow \infty }\mathcal{L}_k(\mu_1, \overline{D}_1)-\mathcal{L}_k(\mu_0,\overline{D}_0)=\frac{1}{\mathrm{vol(D)}}\bigl(\mathcal{E}({\overline{D}_1}_X)-\mathcal{E}({\overline{D}_0}_X)\bigr)=
\mathcal{E}_{\rm{eq}}({\overline{D}_1})-\mathcal{E}_{\rm{eq}}({\overline{D}_0}).
\]
Thus we proved  Theorem \ref{g1} for ample divisors.

\subsection{\scshape{The case of big and nef divisor}\large}

Let $\overline{D}$ a metrized toric divisor such that $D$  is big and nef.
Let $\overline{A}$ be  a positive metrized toric divisor on $X$. Let $l>0$ and  we set
$\phi_l:=P_X(\phi_{\overline{D}}+\frac{1}{l}\phi_A)-\frac{1}{l}\phi_A$ where $\phi_{\overline{D}}$ (resp. $\phi_A$) denotes the associated
weight to $\overline{D}$ (resp. $\overline{A}$). We have $\phi_l$ defines a continuous metric on $D$. Indeed,
this is follows from the fact that $lD+A$ is ample for any $l\in \N$ (because $D$ is nef and $A$ is ample), which
implies that $P_X(l\phi_{\overline{D}}+\phi_A)$
is  a continuous weight on $lD+A$. We denote by $\overline{D}_l'$ the continuous metrized toric divisor $D$
endowed with the continuous weight $\phi_l$.

\begin{lemma}
$(\phi_l)_{l\in \N^\ast}$ is a decreasing sequence of continuous weights on $\mathcal{O}(D)$ converging pointwise to
$P_X\phi$.
\end{lemma}

\begin{proof}
First notice  that the limit of the sequence $(\phi_l)_{l\in \N^\ast}$, if it exists, doesn't depend on the choice
 of the metric on $A$. Indeed,  let $\phi_{1,A}$ and   $\phi_{0,A}$ be
  two  weights on $A$ defining  continuous  metrics and we
 set $\phi_{i,l}:=P_X(\phi_{\overline{D}}+\frac{1}{l}\phi_{i,A})-\frac{1}{l}\phi_{i,A}$ for $i=0,1$. Then by Proposition \ref{w2} we have
\[
|\phi_{1,l}-\phi_{0,l}|\leq \frac{2}{l}\sup_{x\in X}|\phi_{1,A}-\phi_{0,A}|, \quad\forall l\in \N^\ast.
\]
 Suppose that $\phi_A$ is psh. Let $\psi$  be a psh weight on $D$ with $\psi\leq \phi_{\overline{D}}$ then
 $\psi+\frac{1}{l}\phi_A$ is also a psh weight satisfying $\psi+\frac{1}{l}\phi_A\leq P_X(\phi_A+\frac{1}{l}
 \phi_A)$. Therefore, $\psi\leq
 P_X(\phi_{\overline{D}}+\frac{1}{l}\phi_A)-\frac{1}{l}\phi_A$ for any $\psi$ a psh weight on $D$  such that
 $\psi\leq \phi_{\overline{D}}$.  Thus
 \[
 P_X\phi_{\overline{D}}\leq \phi_l\leq \phi_{\overline{D}}\quad \forall l\in \N^\ast.
 \]
 Let $l\geq k$, then clearly $P_X(\phi_{\overline{D}}+\frac{1}{l}\phi_A)+(\frac{1}{k}-\frac{1}{l})\phi_A\leq
 P_X(\phi_{\overline{D}}+\frac{1}{k}\phi_A)$. So
 \[
 \phi_l\leq \phi_{k}\quad \forall l\geq k.
 \]
If we set $\Psi:=\lim_{l\in \N^\ast,l\mapsto \infty}\phi_l$, then
\begin{equation}\label{f4}
P_X\phi\leq \Psi\leq\phi_{l+1}\leq \phi_l\leq  \phi_{\overline{D}}\quad \forall l\in \N^\ast.
\end{equation}
 Let $k\in \N^\ast$, we have $\phi_l+\frac{1}{k}\phi_A$ is psh for any  $  l\geq k$.
Then $\Psi+\frac{1}{k}\phi_A$ is also a psh function for any $k\in \N^\ast$ (see for instance
\cite[Theorem 5.4]{DemaillyLivre}). It follows that
$\Psi$ is psh weight on $D$. We conclude that
\[
P_X\phi_{\overline{D}}=\Psi.
\]
\end{proof}
Recall that $\phi_l$ is continuous for any $l\in \N^\ast$.
 We have $(\phi_l+\frac{1}{k}\phi_A)_{l\geq k}$ is a decreasing sequence of continuous  psh functions converging pointwise to $P_X\phi_{\overline{D}}+\frac{1}{k}\phi_A$
with minimal singularities.  Then as $l$ tends to $\infty$
\begin{equation}\label{w3}
\mathcal{E}(\overline{D}_l'+\frac{1}{k}\overline{A})\rightarrow \mathcal{E}(\overline{D}_X+\frac{1}{k}\overline{A}),
\end{equation}
for any $k\in \N^\ast$. The proof of the last assertion is similar to \cite[Proposition 4.3]{BermanBoucksom} which is  a consequence of
the continuity of the Monge-Amp\`ere operator, see \cite{BEGZ}.\\

Let $\phi_{0,D}$ be a continuous and $\mathbb{S}_Q$-invariant
psh weight on $D$ and we set $\overline{D}'$ the metrized toric $(D,\phi_{0,D})$.
We set $g_{k,l}:=g_{\overline{D}_l'+\frac{1}{k}\overline{A}}$ for any $k,l\in \N^\ast$.  Then by \eqref{eq} and \eqref{f4}
we get
\begin{equation}\label{f5}
\check{g}_{\overline{D}_X+\frac{1}{k}\overline{A}}\leq \check{g}_{k,l+1}\leq \check{g}_{k,l}\leq \check{g}_{\overline{D}+\frac{1}{k}\overline{A}}.
\end{equation} By Theorem \ref{x4} we have
\[
\mathcal{E}(\overline{D}_l'+\frac{1}{k}\overline{A})-\mathcal{E}(\overline{D}'+\frac{1}{k}\phi_A)=
-\int_{\Delta_D+\frac{1}{k}\Delta_A}(\check{g}_{k,l}(x)-\check{g}_{{\overline{D}_0+\frac{1}{k}\overline{A}}}(x))dx.
\]
Now using \eqref{f5}, get  for any $l,k\in \N^\ast$,
\[
-\int_{\Delta_D+\frac{1}{k}\Delta_A} (\check{g}_{\overline{D}+\frac{1}{k}\overline{A}}-\check{g}_{{\overline{D}_0+\frac{1}{k}\overline{A}}})dx
\leq \mathcal{E}(\overline{D}_l+\frac{1}{k}\overline{A})-\mathcal{E}(\overline{D}'+\frac{1}{k}\overline{A})\leq -\int_{\Delta_D+\frac{1}{k}\Delta_A}(\check{g}_{\overline{D}_X+\frac{1}{k}\overline{A}}
-\check{g}_{{\overline{D}_0+\frac{1}{k}\overline{A}}})dx.
\]
As  $l$ tends to $\infty$, we have by \eqref{w3}
\begin{align*}
-\int_{\Delta_D+\frac{1}{k}\Delta_A} (\check{g}_{\overline{D}+\frac{1}{k}\overline{A}}-\check{g}_{{\overline{D}_0+\frac{1}{k}\overline{A}}}(x))dx
\leq \mathcal{E}(\overline{D}_X+\frac{1}{k}\overline{A})-\mathcal{E}(\overline{D}'+\frac{1}{k}\overline{A})\leq -\int_{\Delta_D+\frac{1}{k}\Delta_A}(\check{g}_{\overline{D}_X+\frac{1}{k}\overline{A}}(x)
-\check{g}_{{\overline{D}_0+\frac{1}{k}\overline{A}}}(x))dx,
\end{align*}
for any $k\in \N^\ast$. By Proposition \ref{x5}, we get
\begin{align*}
-\int_{\Delta_D} (\check{g}_{\overline{D}}-\check{g}_{{\overline{D}_0}})dx
&\leq \liminf_k\mathcal{E}(\overline{D}_X+\frac{1}{k}\overline{A})-\mathcal{E}(\overline{D}'+\frac{1}{k}\overline{A})\\
&\leq \limsup_k\mathcal{E}(\overline{D}_X+\frac{1}{k}\overline{A})-\mathcal{E}(\overline{D}'+\frac{1}{k}\overline{A}) \leq -\int_{\Delta_D}(\check{g}_{\overline{D}_X}
-\check{g}_{{\overline{D}_0}})dx.
\end{align*}
Using lemma \ref{eq}, we deduce that
\begin{equation}\label{r3}
 \lim_k\mathcal{E}(\overline{D}_X+\frac{1}{k}\overline{A})-\mathcal{E}(\overline{D}'+\frac{1}{k}\overline{A})= -\int_{\Delta_D} (\check{g}_{\overline{D}}-\check{g}_{{\overline{D}_0}})dx.
\end{equation}

Since $P_X\phi_{\overline{D}}-\phi_{\infty,D}$ is bounded then, by lemma \ref{w4} below, there exists a constant  $C$  such that
\[
\bigl|\mathcal{E}(\overline{D}_X+\frac{1}{k}\overline{A})-\mathcal{E}(\overline{D}'+\frac{1}{k}\overline{A}))-
\mathcal{E}(\overline{D}_X)+\mathcal{E}(\overline{D}')
\bigr|\leq \frac{C}{k}\quad \forall\,k\in \N^\ast.
\]
Then
\[
\liminf_{l\mapsto \infty} \mathcal{E}(\overline{D}_X+\frac{1}{k}\overline{A})-\mathcal{E}(\overline{D}'+\frac{1}{k}\overline{A})
=\mathcal{E}(\overline{D}_X)-\mathcal{E}(\overline{D}').
\]
Combined with \eqref{r3}, we obtain
\[
-\int_{\Delta_D}(\check{g}_{{\overline{D}}_X}(x)-\check{g}_{{\overline{D}_0}}(x))dx=\mathcal{E}({\overline{D}}_X)-\mathcal{E}({\overline{D}_0}).
\]
We conclude that,
\[
\lim_{k\rightarrow \infty }\mathcal{L}_k(\mu_1, \overline{D}_1)-\mathcal{L}_k(\mu_0,\overline{D}_0)=
\mathcal{E}_{\rm{eq}}({\overline{D}_1})-\mathcal{E}_{\rm{eq}}({\overline{D}_0}).
\]
This ends the proof of Theorem  \ref{g1}.

\begin{lemma}\label{w4} Let $D$ be a big and nef divisor  and $A$ an ample
divisor on $X$. Let $\psi_D$ and  $\phi_{0,D}$ be two  psh weight with minimal
 singularities on $D$. Let $\phi_A$ and $\phi_{0,A}$ be two positive continuous weight on $A$. We assume
 that $\phi_{0,D}, \phi_A$ and $\phi_{0,A}$ are continuous and  $\psi_D-\phi_{\infty,D}$ is bounded.
Then there exists a real polynomial $P$ of degree $\leq n$ depending only on $\phi_A,\psi_D,\phi_{0,D}$ such that
\[
\bigl|\mathcal{E}(k\overline{D}_\psi+\overline{A})-\mathcal{E}(k\overline{D}'+\overline{A})-k^{n+1}(\mathcal{E}(\overline{D}_{\psi})-\mathcal{E}(
\overline{D}')
)
\bigr|
\leq P(k)\quad \forall k\in \N,
\]
where $\overline{D}_\psi$, $\overline{D}'$ and $\overline{A}$ are the metrized toric divisors endowed with
$\psi_D$, $\phi_{0,D}$ and $\phi_A$ respectively.
\end{lemma}
\begin{proof}
We have
\begin{align*}
\mathcal{E}(k\overline{D}_\psi+\overline{A})-\mathcal{E}(k\overline{D}'+\overline{A})&=
\sum_{i=0}^n\int_X(k(\psi_D-\phi_{0,D}))dd^c(k\psi_D+\phi_A)^i dd^c(k\phi_{0,D}+\phi_A)^{n-i}\\
&=k^{n+1}(\mathcal{E}(\overline{D}_\psi)-\mathcal{E}(\overline{D}'))+\int_X(\psi_D-\phi_{0,D})T,
\end{align*}
where $T$ is a linear sum of terms of the form $dd^c(\phi_A)^\al dd^c(\phi_{0,D})^\beta dd^c(\psi_D)^\gamma k^\eps$
with $\al,\beta,\gamma, \eps\in \N$ and $\al+\beta+\gamma=n$ and $\eps\leq n$.

We set
$f_D:=\psi_D-\phi_{0,D}$. This function is bounded since we assume that $\psi_D-\phi_{\infty,D}$ is bounded and $\phi_{0,D}$ is continous. Then,
\begin{align*}
\Bigl| \int_X f_D dd^c(\phi_A)^\al dd^c(\phi_{0,D})^\beta dd^c(\psi_D)^\gamma \Bigr|\leq
\sup_{X}|f_D|\int_X c_1(A)^\al c_1(D)^{\beta+ \gamma}.
\end{align*}
The lemma follows from the last inequality.

\end{proof}

\bibliographystyle{plain}

\bibliography{biblio}

\end{document}